\documentclass[british]{article}
\usepackage{ae,aecompl}
\usepackage[T1]{fontenc}
\usepackage[latin9]{luainputenc}
\usepackage{amsmath}
\usepackage{amsthm}
\usepackage{amssymb}

\makeatletter
\theoremstyle{plain}
\newtheorem{thm}{\protect\theoremname}
  \theoremstyle{remark}
  \newtheorem{rem}[thm]{\protect\remarkname}
  \theoremstyle{plain}
  \newtheorem{prop}[thm]{\protect\propositionname}
  \theoremstyle{plain}
  \newtheorem{lem}[thm]{\protect\lemmaname}

\@ifundefined{date}{}{\date{}}
\usepackage{aecompl}

\usepackage{aecompl}

\usepackage{babel}
\addto\captionsbritish{}
  \addto\captionsbritish{\renewcommand{\lemmaname}{Lemma}}
  \addto\captionsbritish{\renewcommand{\propositionname}{Proposition}}
  \addto\captionsbritish{\renewcommand{\theoremname}{Theorem}}
  \addto\captionsenglish{}
  \addto\captionsenglish{\renewcommand{\lemmaname}{Lemma}}
  \addto\captionsenglish{\renewcommand{\propositionname}{Proposition}}
  \addto\captionsenglish{\renewcommand{\theoremname}{Theorem}}
  
  \providecommand{\lemmaname}{Lemma}
  \providecommand{\propositionname}{Proposition}
\providecommand{\theoremname}{Theorem}

\usepackage{babel}

  \providecommand{\lemmaname}{Lemma}
  \providecommand{\propositionname}{Proposition}
  \providecommand{\remarkname}{Remark}
\providecommand{\theoremname}{Theorem}

\newcommand{\IR}{\mathbb{R}}
\newcommand{\IE}{\mathbb{E}}

\usepackage{babel}
\providecommand{\lemmaname}{Lemma}
  \providecommand{\propositionname}{Proposition}
  \providecommand{\remarkname}{Remark}
\providecommand{\theoremname}{Theorem}

\makeatother

\usepackage{babel}
  \providecommand{\lemmaname}{Lemma}
  \providecommand{\propositionname}{Proposition}
  \providecommand{\remarkname}{Remark}
\providecommand{\theoremname}{Theorem}

\begin{document}

\title{Critical Regime in a Curie-Weiss Model with two Groups and Heterogeneous
Coupling}

\author{Werner Kirsch and Gabor Toth\\
 Fakult\"at f\"ur Mathematik und Informatik \\
 FernUniversit\"at Hagen, Germany}
\maketitle
\begin{abstract}
We discuss a Curie-Weiss model with two groups in the critical regime.
This is the region where the central limit theorem does not hold any
more but the mean magnetization still goes to zero as the number of
spins grows. We show that the total magnetization normalized by $N^{3/4}$
converges to a non-trivial distribution which is not Gaussian, just
as in the single-group Curie-Weiss model. 
\end{abstract}

\section{Introduction}

The Curie-Weiss model is a simple model of magnetism. In this model
the spins can take values in $\{-1,1\}$. The energy function for
spins $X=(X_{1},X_{2},\ldots,X_{N})\in\{-1,1\}^{N}$ is given by 
\begin{align}
H(X)~=~H(X_{1},\ldots,X_{N})~:=~-\frac{J}{2N}\,\big(\sum_{j=1}^{N}\,X_{j}\big)^{2}\,,
\end{align}
where $J$ is a positive real number.

The probability of a spin configuration is given by 
\begin{align}
\mathbb{P}\big(X_{1}=x_{1},\ldots,X_{N}=x_{N}\big)~:=~Z^{-1}\;e^{-H(x_{1},\ldots,x_{N})}\label{eq:Pbeta}
\end{align}
where $x_{i}\in\{-1,1\}$ and $Z$ is a normalisation constant which
depends on $N$ and $J$.

The quantity 
\begin{align}
S_{N}~=~\sum_{j=1}^{N}X_{j}
\end{align}
is called the (total) magnetisation. It is well known (see e.\,g.
Ellis \cite{Ellis} or \cite{MM}) that the Curie-Weiss model has
a phase transition at $J=1$ in the following sense 
\begin{align}
\frac{1}{N}\,S_{N}~\Longrightarrow~\frac{1}{2}\,(\delta_{-m(J)}+\delta_{m(J)})\label{eq:lln}
\end{align}
where $\Rightarrow$ denotes convergence in distribution and $\delta_{x}$
the Dirac measure in $x$.

For $J\leq1$ we have $m(J)=0$ while $m(J)>0$ for $J>1$.

Equation (\ref{eq:lln}) is a substitute for the law of large numbers
for i.i.d. random variables.

Moreover, for $J<1$ there is a central limit theorem, i.\,e. 
\begin{align}
\frac{1}{\sqrt{N}}\,S_{N}~\Longrightarrow~\mathcal{N}(0,\frac{1}{1-\beta})
\end{align}

In the `critical' case $J=1$ the correct normalization of the $S_{N}$
is $N^{3/4}$ rather than $N^{1/2}$. In fact the normalized sums
\begin{align}
\frac{1}{N^{3/4}}\,S_{N}
\end{align}
converge in distribution. The limit measure is not a normal distribution.

In this paper we consider a Curie-Weiss model with two groups of spins
denoted by $X=(X_{1},\ldots,X_{\tilde{N}_{1}})$ and $Y=(Y_{1},\ldots,Y_{\tilde{N}_{2}})$
with $X_{i},Y_{j}\in\{-1,1\}$. The total number of spins is $N=\tilde{N}_{1}+\tilde{N}_{2}$.
The interaction within the groups is given by the coupling constants
$J_{1}$ and $J_{2}$ and the interaction between spins from different
groups is $\bar{J}$. In other words, the Hamiltonian is given by:

\begin{align}
H=H(X,Y)~=~-\frac{1}{2N}\left[J_{1}\big(\sum_{j=1}^{\tilde{N}_{1}}X_{j}\big)^{2}+J_{2}\big(\sum_{j=1}^{\tilde{N}_{2}}Y_{j}\big)^{2}+2\bar{J}\sum_{i=1}^{\tilde{N}_{1}}\sum_{j=1}^{\tilde{N}_{2}}X_{i}Y_{j}\right].
\end{align}

We assume that $J_{1},J_{2}>0$, $\bar{J}\geq0$ and $J_{1}J_{2}-{\bar{J}}^{2}>0$,
so that the coupling matrix

\begin{align}
J=\left[\begin{array}{cc}
J_{1} & \bar{J}\\
\bar{J} & J_{2}
\end{array}\right]
\end{align}
is positive definite. We denote its determinant by $\Delta=J_{1}J_{2}-\bar{J}^{2}>0$.

Let us consider groups $X_{1},\ldots,X_{N_{1}}$ and $Y_{1},\ldots,Y_{N_{2}}$
(with $N_{1}\leq\tilde{N}_{1}$ and $N_{2}\leq\tilde{N}_{2}$ such
that both $N_{1}$ and $N_{2}$ go to infinity as $N$ does, more
precisely, we set 
\begin{equation}
\alpha_{1}=\lim\frac{N_{1}}{N}\qquad\qquad\alpha_{2}=\lim_{N\to\infty}\frac{N_{2}}{N}\,
\end{equation}
and assume that these limits exist and are strictly positive.

We are interested in the asymptotic behaviour of the two-dimensional
random variables 
\begin{equation}
S_{N_{1}\,N_{2}}~=~(S_{N_{1}}^{(1)},(S_{N_{2}}^{(2)})~:=~\big(\sum_{i=1}^{N_{1}}\,X_{i}\,,\,\sum_{j=1}^{N_{2}}\,Y_{j}\,\big)
\end{equation}
as $N$ goes to infinity.

In this paper we consider what we call the `critical phase', i.\,e.
the regime where 
\begin{align}
J_{1}~ & <~\frac{1}{\alpha_{1}},\label{eq:condJ1-1}\\
J_{2}~ & <~\frac{1}{\alpha_{2}},\label{eq:condJ2-1}\\
\bar{J}^{2}~ & =~(\frac{1}{\alpha_{1}}-J_{1})(\frac{1}{\alpha_{2}}-J_{2}).\label{eq:condJ3-1}
\end{align}

Note that if we use the symbol $\alpha$ for the diagonal $2\times2$
matrix with entries $\alpha_{1}$ and $\alpha_{2}$, we can formulate
these conditions equivalently in matrix form: the matrix

\[
J^{-1}-\alpha
\]
is singular and has positive diagonal entries if and only if we are
in the critical regime.

For later use we define the matrix $L$ by

\begin{align}
L=\left[\begin{array}{cc}
L_{1} & -\bar{L}\\
-\bar{L} & L_{2}
\end{array}\right]=\frac{1}{J_{1}J_{2}-\bar{J}^{2}}\left[\begin{array}{cc}
J_{2} & -\bar{J}\\
-\bar{J} & J_{1}
\end{array}\right]=J^{-1}.\label{eq:defL}
\end{align}

In the previous paper \cite{KT2} we discussed the `high temperature
regime' for the model under consideration for which \eqref{eq:condJ3-1}
is replaced with 
\begin{align}
\bar{J}^{2}~ & <~(\frac{1}{\alpha_{1}}-J_{1})(\frac{1}{\alpha_{2}}-J_{2}).\label{eq:condJ3-10}
\end{align}
For this range of parameters we proved that 
\begin{align}
\Big(\frac{1}{N_{1}}\,S_{N_{1}}^{(1)},\frac{1}{N_{2}}\,S_{N_{2}}^{(2)}\Big)~~\Longrightarrow~~\delta_{(0,0)}\,.\label{eq:lln1}
\end{align}

For the `high temperature regime' we also proved a central limit theorem,
namely 
\begin{align}
\Big(\frac{1}{\sqrt{N_{1}}}\,S_{N_{1}}^{(1)},\frac{1}{\sqrt{N_{2}}}\,S_{N_{2}}^{(2)}\Big)\label{eq:cltht}
\end{align}
converges to a normal distribution (for the details see \cite{KT2}).

In the critical regime we consider here \big(i.\,e. for \eqref{eq:condJ1-1}\textendash \eqref{eq:condJ3-1}\big)
we will prove that \eqref{eq:lln1} still holds but \eqref{eq:cltht}
has to be replaced by 
\begin{align}
T_{N}~=~\Big(\frac{1}{{N_{1}}^{3/4}}\,S_{N_{1}}^{(1)},\frac{1}{{N_{2}}^{3/4}}\,S_{N_{2}}^{(2)}\Big)\,.\label{eq:cltcrit}
\end{align}
This sequence $T_{N}$ converges in distribution but not to a normal
distribution. We compute the moments of the limiting measure in Theorem
\ref{Fluctuations}.

Now we are able to formulate our results. 
\begin{thm}
\label{LLN} Under the above assumptions, we have

\[
\big(\frac{1}{N_{1}}\sum_{i=1}^{N_{1}}X_{i},\frac{1}{N_{2}}\sum_{j=1}^{N_{2}}Y_{j}\big)~\underset{N\to\infty}{\Longrightarrow}\delta_{(0,0)}.
\]
\end{thm}

Above '$\Longrightarrow$' denotes convergence in distribution of
the $2$-dimensional random variable on the left hand side.

If we choose as normalising factors $N_{\nu}^{\frac{3}{4}}$ instead
of $N_{\nu}$, then we obtain 
\begin{thm}
\label{Fluctuations} Under the above assumptions, the random variables

\[
(\frac{1}{N_{1}^{3/4}}\sum_{i=1}^{N_{1}}X_{i},\frac{1}{N_{2}^{3/4}}\sum_{j=1}^{N_{2}}Y_{j})
\]

converge in distribution to a measure $\mu$ (on $\IR^{2}$) with
moments

\begin{align}
 & m_{KL}~:=~\int x^{K}y^{L}\;d\mu(x,y)\\
~ & =~\left[\frac{12}{\alpha_{1}(L_{2}-\alpha_{2})^{2}+\alpha_{2}(L_{1}-\alpha_{1})^{2}}\right]^{\frac{K+L}{4}}(L_{1}-\alpha_{1})^{\frac{L}{2}}(L_{2}-\alpha_{2})^{\frac{K}{2}}\frac{\Gamma(\frac{K+L+1}{4})}{\Gamma(\frac{1}{4})}\alpha_{1}^{\frac{K}{4}}\alpha_{2}^{\frac{L}{4}}.
\end{align}
\end{thm}

\begin{rem}
For the special case $J_{1}=J_{2}=J$, $\alpha_{1}=\alpha_{2}=1/2$,
and $J+\bar{J}=2$, the moments are approximately

\[
12^{\frac{K+L}{4}}\frac{\Gamma(\frac{K+L+1}{4})}{\Gamma(\frac{1}{4})}\alpha_{1}^{\frac{K}{4}}\alpha_{2}^{\frac{L}{4}}.
\]

These moments are identical to those for the model with homogeneous
coupling matrix and $J=1$ (cf. theorem 14 in \cite{KT1}). 
\end{rem}

We prove the two theorems in section \ref{sec:Proofs}. The proof
uses concepts and notations from \cite{KT1,KT2}. For the readers
convenience, we briefly review them in section \ref{sec:Preparation}.

We mention that the Curie-Weiss model is also used to model the behaviour
of voters who have the choice to vote `Yea' (spin=1, say) or `Nay'
(spin=-1) (see \cite{Penrose}).

While finishing this paper we became aware of the papers \cite{FM}
and \cite{FC} which contain the above results as special cases. The
methods used by those authors is very different from ours.

We are grateful to Francesca Collet for drawing our attention to the
papers \cite{FM} and \cite{FC}.

We would also like to thank Matthias L\"owe and Kristina Schubert for
valuable discussion and for making their preprint \cite{LS} available
prior to publication.

\section{\label{sec:Preparation}Preparation}

In the proof of the results we employ the moment method (see e.\,g.
\cite{Breiman} or \cite{MM}). This technique was already employed
in our papers \cite{KT1} and \cite{KT2}. We use the notation introduced
there and refer the reader for details to these sources.

To use the method of moments we have to evaluate sums of the form
\begin{align}
 & \mathbb{E}\left[\left(\sum_{i=1}^{N_{1}}X_{i}\right)^{K}\left(\sum_{j=1}^{N_{2}}Y_{j}\right)^{L}\right]\nonumber \\
=~ & \sum_{i_{1},\ldots,i_{K}}\sum_{j_{1},\ldots,j_{L}}\,\mathbb{E}\Big(X_{i_{1}}\cdot X_{i_{2}}\cdot\ldots\cdot X_{i_{K}}\cdot Y_{j_{1}}\cdot Y_{j_{2}}\cdot\ldots\cdot Y_{j_{L}}\Big)\,.
\end{align}

To each $K$-tuple $(i_{1},i_{2},\ldots,i_{K})\in\{1,2,\ldots,N\}^{K}$
we associate a profile $r=(r_{1},r_{2},\ldots,r_{K})$ where $r_{m}$
counts the number of different indices $i_{\nu}$ which occur exactly
$m$ times in this product.

Observe that 
\begin{align}
\IE\Big(X_{i_{1}}\cdot X_{i_{2}}\cdot\ldots\cdot X_{i_{K}}\,\cdot\,Y_{j_{1}}\cdot Y_{j_{2}}\cdot\ldots\cdot Y_{j_{L}}\Big)
\end{align}
depends only on the profiles $r$ and $s$ of the tuples $(i_{1},i_{2},\ldots,i_{K})$
and $(j_{1},j_{2},\ldots,j_{L})$. So we may and will write 
\begin{align}
\IE\Big(X(r)\,Y(s)\Big)~=~\IE\Big(X_{i_{1}}\cdot X_{i_{2}}\cdot\ldots\cdot X_{i_{K}}\,\cdot\,Y_{j_{1}}\cdot Y_{j_{2}}\cdot\ldots\cdot Y_{j_{L}}\Big)
\end{align}
whenever $r$ and $s$ are the profiles of $(i_{1},\ldots,i_{K})$
and $(j_{1},\ldots,j_{L})$ respectively.

We denote by $\Pi^{(K)}$ the set of all profiles of $K$-tuples $(i_{1},\ldots,i_{K})\in\{1,\ldots,N\}^{K}$.

For each profile $r$ of a $K$-tuple $(i_{1},i_{2},\ldots,i_{K})\in\{1,2,\ldots,N\}^{K}$
let us denote by $w_{K}(r)$ the number of tuples in $\{1,2,\ldots,N\}^{K}$
with that profile. It turns out that 
\begin{align}
w_{K}(r)=\frac{N!}{r_{1}!r_{2}!\ldots r_{K}!(N-\sum r_{i})!}\frac{K!}{1!^{r_{1}}2!^{r_{2}}\cdots L!^{r_{K}}}.
\end{align}

With these notations we have 
\begin{align}
 & \sum_{i_{1},\ldots,i_{K}}\sum_{j_{1},\ldots,j_{L}}\,\mathbb{E}\Big(X_{i_{1}}\cdot X_{i_{2}}\cdot\ldots\cdot X_{i_{K}}\cdot Y_{j_{1}}\cdot Y_{j_{2}}\cdot\ldots\cdot Y_{j_{L}}\Big)\\
= & \sum_{r\in\Pi^{(K)}}\,\sum_{s\in\Pi^{(L)}}w_{K}(r)\,w_{L}(s)\;\IE\Big(X(r)\cdot Y(s)\Big)
\end{align}

We define the inverse matrix

\[
L=\left[\begin{array}{cc}
L_{1} & \bar{L}\\
\bar{L} & L_{2}
\end{array}\right]=\frac{1}{J_{1}J_{2}-\bar{J}^{2}}\left[\begin{array}{cc}
J_{2} & -\bar{J}\\
-\bar{J} & J_{1}
\end{array}\right]=J^{-1}.
\]

In order to calculate the correlations $\mathbb{E}\Big(X_{i_{1}}\cdot\ldots\cdot X_{i_{K}}\cdot Y_{j_{1}}\cdot\ldots\cdot Y_{j_{L}}\Big)$
in the sum above, we estimate an integral

\[
\mathcal{Z}_{N}(K,L):=\int e^{-NF(y_{1},y_{2})}\tanh^{K}y_{1}\tanh^{L}y_{2}\mathrm{d}^{2}y,
\]
where the function $F$ is given by

\[
F(y_{1},y_{2})=\frac{1}{2}L_{1}y_{1}^{2}+\frac{1}{2}L_{2}y_{2}^{2}+\bar{L}y_{1}y_{2}-\alpha_{1}\ln\cosh y_{1}-\alpha_{2}\ln\cosh y_{2}.
\]

We apply Laplace's Method to estimate the integral and therefore have
to determine the minima of $F$. For a complete explanation of the
procedure, see section 3 in \cite{KT2}.

\section{\label{sec:Proofs}Proofs}
\begin{prop}
If $L_{\nu}-\alpha_{\nu}>0$ for both groups and $(L_{1}-\alpha_{1})(L_{2}-\alpha_{2})=\bar{L}^{2}$,
then the function F defined in the last section has a unique global
minimum at the origin. 
\end{prop}

\begin{rem}
The conditions stated in the proposition are equivalent to the critical
regime. This is shown in analogous fashion to the proof of proposition
10 in \cite{KT2}. 
\end{rem}

\begin{proof}
We take derivatives of $F$ with respect to both variables

\begin{align*}
F_{1}(y_{1},y_{2}) & =L_{1}y_{1}-\bar{L}y_{2}-\alpha_{1}\tanh y_{1}=0,\\
F_{2}(y_{1},y_{2}) & =L_{2}y_{2}-\bar{L}y_{1}-\alpha_{2}\tanh y_{2}=0.
\end{align*}

One solution to this system of equations is $y_{1}=y_{2}=0$. We proceed
to show that this solution is unique. We rewrite the function $F$:

\[
F(tx_{0},ty_{0})=\frac{1}{2}L_{1}t^{2}x_{0}^{2}+\frac{1}{2}L_{2}t^{2}y_{0}^{2}-\bar{L}x_{0}y_{0}t^{2}-\alpha_{1}\ln\cosh tx_{0}-\alpha_{2}\ln\cosh ty_{0},
\]
where $(x_{0},y_{0})$ indicates the direction, $x_{0}^{2}+y_{0}^{2}=1$,
and $t$ is the distance from the origin. The first derivative of
$F$ with respect to $t$ is 0 at the origin, independently of the
direction $(x_{0},y_{0})$ .

We show that the second derivative $\frac{\mathrm{d}^{2}F(tx_{0},ty_{0})}{\mathrm{d}t^{2}}$
is positive in all directions, except for two.

\begin{align*}
\frac{\mathrm{d}^{2}F(tx_{0},ty_{0})}{\mathrm{d}t^{2}} & =L_{1}x_{0}^{2}+L_{2}y_{0}^{2}-2\bar{L}x_{0}y_{0}-\frac{\alpha_{1}x_{0}^{2}}{\cosh^{2}tx_{0}}-\frac{\alpha_{2}y_{0}^{2}}{\cosh^{2}ty_{0}}.
\end{align*}

Therefore, we have

\begin{align*}
\left.\frac{\mathrm{d}^{2}F(tx_{0},ty_{0})}{\mathrm{d}t^{2}}\right|_{t=0} & =L_{1}x_{0}^{2}+L_{2}y_{0}^{2}-2\bar{L}x_{0}y_{0}-\alpha_{1}x_{0}^{2}-\alpha_{2}y_{0}^{2}\\
 & =(\sqrt{L_{1}-\alpha_{1}}x_{0}-\sqrt{L_{2}-\alpha_{2}}y_{0})^{2}+2(\sqrt{L_{1}-\alpha_{1}}\sqrt{L_{2}-\alpha_{2}}-\bar{L})x_{0}y_{0}\\
 & =(\sqrt{L_{1}-\alpha_{1}}x_{0}-\sqrt{L_{2}-\alpha_{2}}y_{0})^{2}\\
 & \geq0,
\end{align*}
with equality if and only if

\[
\sqrt{L_{1}-\alpha_{1}}x_{0}-\sqrt{L_{2}-\alpha_{2}}y_{0}=0.
\]
Hence there are two directions $(x_{0},y_{0})$, one pointing into
quadrant one, the other into quadrant three, in which the second derivative
is 0 at the origin. In all other directions the second derivative
is strictly positive. For any direction, the second derivative is
strictly positive for all $t>0$:

\begin{align*}
\frac{\mathrm{d}^{2}F(tx_{0},ty_{0})}{\mathrm{d}t^{2}} & =L_{1}x_{0}^{2}+L_{2}y_{0}^{2}-2\bar{L}x_{0}y_{0}-\frac{\alpha_{1}x_{0}^{2}}{\cosh^{2}tx_{0}}-\frac{\alpha_{2}y_{0}^{2}}{\cosh^{2}ty_{0}}\\
 & >L_{1}x_{0}^{2}+L_{2}y_{0}^{2}-2\bar{L}x_{0}y_{0}-\alpha_{1}x_{0}^{2}-\alpha_{2}y_{0}^{2}\\
 & =(\sqrt{L_{1}-\alpha_{1}}x_{0}-\sqrt{L_{2}-\alpha_{2}}y_{0})^{2}+2(\sqrt{L_{1}-\alpha_{1}}\sqrt{L_{2}-\alpha_{2}}-\bar{L})x_{0}y_{0}\\
 & =(\sqrt{L_{1}-\alpha_{1}}x_{0}-\sqrt{L_{2}-\alpha_{2}}y_{0})^{2}\\
 & \geq0.
\end{align*}

This concludes the proof that the minimum at the origin is unique
and global. 
\end{proof}
\begin{thm}
\label{thm:correlations}Let $L_{\nu}-\alpha_{\nu}>0$ for both groups
and $(L_{1}-\alpha_{1})(L_{2}-\alpha_{2})=\bar{L}^{2}$. Then for
all $K,L\in\mathbb{N}_{0}$, the expected value $\mathbb{E}(X_{1}\cdots X_{K}Y_{1}\cdots Y_{L})$
is asymptotically given by the expression

\[
\left[\frac{12}{\alpha_{1}(L_{2}-\alpha_{2})^{2}+\alpha_{2}(L_{1}-\alpha_{1})^{2}}\right]^{\frac{K+L}{4}}(L_{1}-\alpha_{1})^{\frac{L}{2}}(L_{2}-\alpha_{2})^{\frac{K}{2}}\frac{\Gamma(\frac{K+L+1}{4})}{\Gamma(\frac{1}{4})}\frac{1}{N^{\frac{K+L}{4}}}
\]

if both $K+L$ is even and zero otherwise. 
\end{thm}

\begin{rem}
For the special case $J_{1}=J_{2}=J$, $\alpha_{1}=\alpha_{2}=1/2$,
and $J+\bar{J}=2$, the correlations are approximately

\[
12^{\frac{K+L}{4}}\frac{\Gamma(\frac{K+L+1}{4})}{\Gamma(\frac{1}{4})}\frac{1}{N^{\frac{K+L}{4}}}.
\]

These correlations are identical to those for the model with homogeneous
coupling matrix and $\beta=1$. 
\end{rem}

\begin{proof}
As the Hessian matrix is singular at the origin, we need higher order
terms in our Taylor polynomial for $F$. We calculate the third and
fourth order derivatives:

\begin{align*}
F_{11}(x,y) & =L_{1}-\frac{\alpha_{1}}{\cosh^{2}x},\\
F_{12}(x,y) & =-\bar{L},\\
F_{22}(x,y) & =L_{2}-\frac{\alpha_{2}}{\cosh^{2}y},\\
F_{111}(x,y) & =2\alpha_{1}\frac{\tanh x}{\cosh^{2}x},\\
F_{222}(x,y) & =2\alpha_{2}\frac{\tanh y}{\cosh^{2}y},\\
F_{1111}(x,y) & =\frac{2\alpha_{1}}{\cosh^{2}x}\left(\frac{1}{\cosh^{2}x}-2\tanh^{2}x\right),\\
F_{2222}(x,y) & =\frac{2\alpha_{2}}{\cosh^{2}y}\left(\frac{1}{\cosh^{2}y}-2\tanh^{2}y\right).
\end{align*}

All other third and fourth order derivatives are 0. At the origin,
we have the following values

\begin{align*}
F_{11}(0,0) & =L_{1}-\alpha_{1},\\
F_{12}(0,0) & =-\bar{L},\\
F_{22}(0,0) & =L_{2}-\alpha_{2},\\
F_{111}(0,0) & =0,\\
F_{222}(0,0) & =0,\\
F_{1111}(0,0) & =2\alpha_{1},\\
F_{2222}(0,0) & =2\alpha_{2}.
\end{align*}

Then the Taylor polynomial of order four reads

\begin{align*}
 & \frac{1}{2}\left[(L_{1}-\alpha_{1})x^{2}+(L_{2}-\alpha_{2})y^{2}-2\bar{L}xy+\frac{2\alpha_{1}}{12}x^{4}+\frac{2\alpha_{2}}{12}y^{4}\right]\\
= & \frac{1}{2}\left[(L_{1}-\alpha_{1})x^{2}+(L_{2}-\alpha_{2})y^{2}-2\sqrt{L_{1}-\alpha_{1}}\sqrt{L_{2}-\alpha_{2}}xy+\frac{\alpha_{1}}{6}x^{4}+\frac{\alpha_{2}}{6}y^{4}\right]\\
= & \frac{1}{2}\left[(\sqrt{L_{1}-\alpha_{1}}x-\sqrt{L_{2}-\alpha_{2}}y)^{2}+\frac{\alpha_{1}}{6}x^{4}+\frac{\alpha_{2}}{6}y^{4}\right].
\end{align*}

We need to estimate the integral

\[
\mathcal{Z}_{N}(K,L)=\int_{\mathbb{R}^{2}}e^{-\frac{N}{2}\left[(\sqrt{L_{1}-\alpha_{1}}x-\sqrt{L_{2}-\alpha_{2}}y)^{2}+\frac{\alpha_{1}}{6}x^{4}+\frac{\alpha_{2}}{6}y^{4}\right]}x^{K}y^{L}\mathrm{d}x\mathrm{d}y.
\]
We start by substituting

\begin{align*}
u & :=\sqrt{L_{1}-\alpha_{1}}x-\sqrt{L_{2}-\alpha_{2}}y,\\
v & :=\sqrt{L_{1}-\alpha_{1}}x+\sqrt{L_{2}-\alpha_{2}}y.
\end{align*}
Then the above integral is

\[
\int_{\mathbb{R}^{2}}e^{-\frac{N}{2}\left[u^{2}+\frac{\alpha_{1}}{6}(\frac{u+v}{2\sqrt{L_{1}-\alpha_{1}}})^{4}+\frac{\alpha_{2}}{6}(\frac{v-u}{2\sqrt{L_{2}-\alpha_{2}}})^{4}\right]}(\frac{u+v}{2\sqrt{L_{1}-\alpha_{1}}})^{K}(\frac{v-u}{2\sqrt{L_{2}-\alpha_{2}}})^{L}\mathrm{d}u\mathrm{d}v
\]
divided by the determinant of the Jacobi matrix $2\sqrt{L_{1}-\alpha_{1}}\sqrt{L_{2}-\alpha_{2}}$.
Stripping the integrand of all multiplicative constants, we obtain

\[
\int_{\mathbb{R}^{2}}e^{-\frac{N}{2}\left[u^{2}+\frac{\alpha_{1}}{2^{5}\cdot3(L_{1}-\alpha_{1})^{2}}(u+v)^{4}+\frac{\alpha_{2}}{2^{5}\cdot3(L_{2}-\alpha_{2})^{2}}(v-u)^{4}\right]}(u+v)^{K}(v-u)^{L}\mathrm{d}u\mathrm{d}v.
\]
Now we switch variables again:

\begin{align*}
u' & :=N^{1/2}u,\\
v' & :=N^{1/4}v,
\end{align*}
and obtain

\begin{align*}
\int_{\mathbb{R}^{2}}e^{-\frac{1}{2}\left[u'^{2}+\frac{\alpha_{1}}{2^{5}\cdot3(L_{1}-\alpha_{1})^{2}}(\frac{u'}{N^{1/4}}+v')^{4}+\frac{\alpha_{2}}{2^{5}\cdot3(L_{2}-\alpha_{2})^{2}}(v'-\frac{u'}{N^{1/4}})^{4}\right]}\cdot\\
\cdot(\frac{u'}{N^{1/2}}+\frac{v'}{N^{1/4}})^{K}(\frac{v'}{N^{1/4}}-\frac{u'}{N^{1/2}})^{L}\mathrm{d}u'\mathrm{d}v'
\end{align*}
times a constant equal to

\[
\frac{1}{2^{K+L+1}(L_{1}-\alpha_{1})^{\frac{K+1}{2}}(L_{2}-\alpha_{2})^{\frac{L+1}{2}}N^{\frac{3}{4}}}.
\]
We once again factor out all constants from the integrand:

\begin{align}
\int_{\mathbb{R}^{2}}e^{-\frac{1}{2}\left[u'^{2}+\frac{\alpha_{1}}{2^{5}\cdot3(L_{1}-\alpha_{1})^{2}}(\frac{u'}{N^{1/4}}+v')^{4}+\frac{\alpha_{2}}{2^{5}\cdot3(L_{2}-\alpha_{2})^{2}}(v'-\frac{u'}{N^{1/4}})^{4}\right]}\cdot\label{eq:dom_integral}
\end{align}

\[
\cdot(\frac{u'}{N^{1/4}}+v')^{K}(v'-\frac{u'}{N^{1/4}})^{L}\mathrm{d}u'\mathrm{d}v'.
\]

The above integral is asymptotically equal to

\[
\int_{\mathbb{R}^{2}}e^{-\frac{1}{2}\left[u'^{2}+\frac{v'^{4}}{2^{5}\cdot3}(\frac{\alpha_{1}}{(L_{1}-\alpha_{1})^{2}}+\frac{\alpha_{2}}{(L_{2}-\alpha_{2})^{2}})\right]}v'^{K+L}\mathrm{d}u'\mathrm{d}v',
\]
where we used the dominated convergence theorem (see lemma \ref{lem:dom_conv}).
This integral is equal to the product of

\[
\int_{\mathbb{R}}e^{-\frac{1}{2}u{}^{2}}\mathrm{d}u
\]
and

\[
\int_{\mathbb{R}}e^{-\frac{v^{4}}{2^{6}\cdot3}(\frac{\alpha_{1}}{(L_{1}-\alpha_{1})^{2}}+\frac{\alpha_{2}}{(L_{2}-\alpha_{2})^{2}})}v{}^{K+L}\mathrm{d}v.
\]
The first integral is equal to $\sqrt{2\pi}$. The second has a value
of 0 if $K+L$ is odd, otherwise it is essentially a value of the
gamma function. We set

\begin{align*}
c & :=\frac{1}{2^{6}\cdot3}(\frac{\alpha_{1}}{(L_{1}-\alpha_{1})^{2}}+\frac{\alpha_{2}}{(L_{2}-\alpha_{2})^{2}}),\\
k & :=K+L.
\end{align*}
Then the above integral is

\[
2\int_{0}^{\infty}e^{-cv^{4}}v^{k}\mathrm{d}v.
\]
We substitute

\[
t:=cv^{4}
\]
and calculate

\begin{align*}
 & \frac{2}{4c}\int_{0}^{\infty}e^{-t}(\frac{t}{c})^{\frac{k-3}{4}}\mathrm{d}t\\
= & \frac{1}{2c^{\frac{k+1}{4}}}\int_{0}^{\infty}e^{-t}t^{\frac{k-3}{4}}\mathrm{d}t\\
= & \frac{1}{2c^{\frac{k+1}{4}}}\Gamma(\frac{k+1}{4}).
\end{align*}

This shows that $\mathcal{Z}_{N}(K,L)$ is approximately equal to

\begin{align*}
\frac{\sqrt{2\pi}}{2^{K+L+1}(L_{1}-\alpha_{1})^{\frac{K+1}{2}}(L_{2}-\alpha_{2})^{\frac{L+1}{2}}N^{\frac{K+L+3}{4}}}\frac{1}{2c^{\frac{K+L+1}{4}}}\Gamma(\frac{K+L+1}{4}).
\end{align*}

The correlations $\mathbb{E}(X_{1}\cdots X_{K}Y_{1}\cdots Y_{L})$
are given by

\begin{align*}
\frac{\mathcal{Z}_{N}(K,L)}{\mathcal{Z}_{N}(0,0)} & \approx\frac{1}{2^{K+L}(L_{1}-\alpha_{1})^{\frac{K}{2}}(L_{2}-\alpha_{2})^{\frac{L}{2}}}\frac{1}{c^{\frac{K+L}{4}}}\frac{\Gamma(\frac{K+L+1}{4})}{\Gamma(\frac{1}{4})}\frac{1}{N^{\frac{K+L}{4}}}\\
 & =\frac{1}{2^{K+L}(L_{1}-\alpha_{1})^{\frac{K}{2}}(L_{2}-\alpha_{2})^{\frac{L}{2}}}\left[\frac{2^{6}\cdot3}{\frac{\alpha_{1}}{(L_{1}-\alpha_{1})^{2}}+\frac{\alpha_{2}}{(L_{2}-\alpha_{2})^{2}}}\right]^{\frac{K+L}{4}}\cdot\\
 & \cdot\frac{\Gamma(\frac{K+L+1}{4})}{\Gamma(\frac{1}{4})}\frac{1}{N^{\frac{K+L}{4}}}\\
 & =\frac{1}{(L_{1}-\alpha_{1})^{\frac{K}{2}}(L_{2}-\alpha_{2})^{\frac{L}{2}}}\left[\frac{12(L_{1}-\alpha_{1})^{2}(L_{2}-\alpha_{2})^{2}}{\alpha_{1}(L_{2}-\alpha_{2})^{2}+\alpha_{2}(L_{1}-\alpha_{1})^{2}}\right]^{\frac{K+L}{4}}\cdot\\
 & \cdot\frac{\Gamma(\frac{K+L+1}{4})}{\Gamma(\frac{1}{4})}\frac{1}{N^{\frac{K+L}{4}}}\\
 & =\left[\frac{12}{\alpha_{1}(L_{2}-\alpha_{2})^{2}+\alpha_{2}(L_{1}-\alpha_{1})^{2}}\right]^{\frac{K+L}{4}}(L_{1}-\alpha_{1})^{\frac{L}{2}}(L_{2}-\alpha_{2})^{\frac{K}{2}}\cdot\\
 & \cdot\frac{\Gamma(\frac{K+L+1}{4})}{\Gamma(\frac{1}{4})}\frac{1}{N^{\frac{K+L}{4}}}.
\end{align*}
\end{proof}
\begin{lem}
\label{lem:dom_conv}The integrand in \eqref{eq:dom_integral} is
dominated by

\begin{equation}
\sum_{k=0}^{K}\sum_{l=0}^{L}\left(\begin{array}{c}
K\\
k
\end{array}\right)\left(\begin{array}{c}
L\\
l
\end{array}\right)e^{-\frac{1}{2}\left[u^{2}+av{}^{4}\right]}|u|{}^{L+k-l}|v|{}^{K+l-k},\label{eq:dom_function}
\end{equation}
for some $a>0$, which is itself an integrable function. 
\end{lem}

\begin{proof}
We show the integrand in \eqref{eq:dom_integral} is smaller or equal
to the expression \eqref{eq:dom_function}. The terms $(\frac{u}{N^{1/4}}+v)^{K}(v-\frac{u}{N^{1/4}})^{L}$
can be expanded

\begin{align*}
 & \sum_{k=0}^{K}\sum_{l=0}^{L}\left(\begin{array}{c}
K\\
k
\end{array}\right)\left(\begin{array}{c}
L\\
l
\end{array}\right)(\frac{u}{N^{1/4}})^{k}v{}^{K-k}v{}^{l}(-\frac{u}{N^{1/4}})^{L-l}\\
\leq & \sum_{k=0}^{K}\sum_{l=0}^{L}\left(\begin{array}{c}
K\\
k
\end{array}\right)\left(\begin{array}{c}
L\\
l
\end{array}\right)|u|^{L+k-l}|v|^{K+l-k}.
\end{align*}
The argument of the exponential function in \eqref{eq:dom_integral}
if we ignore the term $u^{2}$ and common multiplicative factors is

\[
\frac{\alpha_{1}}{(L_{1}-\alpha_{1})^{2}}(\frac{u}{N^{1/4}}+v)^{4}+\frac{\alpha_{2}}{(L_{2}-\alpha_{2})^{2}}(v-\frac{u}{N^{1/4}})^{4}.
\]
For a given $v\in\mathbb{R}$, we define the function $g_{v}:\mathbb{R}\rightarrow\mathbb{R}$

\[
g_{v}(u)=\frac{\alpha_{1}}{(L_{1}-\alpha_{1})^{2}}(\frac{u}{N^{1/4}}+v)^{4}+\frac{\alpha_{2}}{(L_{2}-\alpha_{2})^{2}}(v-\frac{u}{N^{1/4}})^{4}.
\]

In order to find the minima of this function, we calculate the first
two derivatives

\begin{align*}
g'_{v}(u) & =\frac{4}{N^{1/4}}\left(\frac{\alpha_{1}}{(L_{1}-\alpha_{1})^{2}}(\frac{u}{N^{1/4}}+v)^{3}-\frac{\alpha_{2}}{(L_{2}-\alpha_{2})^{2}}(v-\frac{u}{N^{1/4}})^{3}\right),\\
g''_{v}(u) & =\frac{12}{N^{1/2}}\left(\frac{\alpha_{1}}{(L_{1}-\alpha_{1})^{2}}(\frac{u}{N^{1/4}}+v)^{2}+\frac{\alpha_{2}}{(L_{2}-\alpha_{2})^{2}}(v-\frac{u}{N^{1/4}})^{2}\right).
\end{align*}
The first derivative is equal to 0 if and only if

\begin{align*}
\frac{\alpha_{1}}{(L_{1}-\alpha_{1})^{2}}(\frac{u}{N^{1/4}}+v)^{3} & =\frac{\alpha_{2}}{(L_{2}-\alpha_{2})^{2}}(v-\frac{u}{N^{1/4}})^{3}\iff\\
\left(\frac{\alpha_{1}}{(L_{1}-\alpha_{1})^{2}}\right)^{\frac{1}{3}}(\frac{u}{N^{1/4}}+v) & =\left(\frac{\alpha_{2}}{(L_{2}-\alpha_{2})^{2}}\right)^{\frac{1}{3}}(v-\frac{u}{N^{1/4}})\iff\\
u_{0} & =vN^{\frac{1}{4}}\frac{\left(\frac{\alpha_{2}}{(L_{2}-\alpha_{2})^{2}}\right)^{\frac{1}{3}}-\left(\frac{\alpha_{1}}{(L_{1}-\alpha_{1})^{2}}\right)^{\frac{1}{3}}}{\left(\frac{\alpha_{1}}{(L_{1}-\alpha_{1})^{2}}\right)^{\frac{1}{3}}+\left(\frac{\alpha_{2}}{(L_{2}-\alpha_{2})^{2}}\right)^{\frac{1}{3}}}.
\end{align*}
Substituting this critical point $u_{0}$ into the second derivative
$g''_{v}$, we notice that $g''_{v}(u_{0})>0$ if $v\neq0$. Hence
in that case, $u_{0}$ is a local minimum. Since the function $g_{v}$
is strictly convex, it is also the only minimum of the function. If
$v=0$, then $u_{0}=0$ is clearly also the global and unique minimum
of $g_{v}$. This shows that the argument of the exponential function
in \eqref{eq:dom_integral} is bounded above by

\[
-\frac{1}{2}\left[u^{2}+av{}^{4}\right]
\]
if we set

\begin{align*}
a & :=\frac{\alpha_{1}}{(L_{1}-\alpha_{1})^{2}}\left(1+\frac{\left(\frac{\alpha_{2}}{(L_{2}-\alpha_{2})^{2}}\right)^{\frac{1}{3}}-\left(\frac{\alpha_{1}}{(L_{1}-\alpha_{1})^{2}}\right)^{\frac{1}{3}}}{\left(\frac{\alpha_{1}}{(L_{1}-\alpha_{1})^{2}}\right)^{\frac{1}{3}}+\left(\frac{\alpha_{2}}{(L_{2}-\alpha_{2})^{2}}\right)^{\frac{1}{3}}}\right)^{\frac{1}{4}}\\
 & +\frac{\alpha_{2}}{(L_{2}-\alpha_{2})^{2}}\left(1-\frac{\left(\frac{\alpha_{2}}{(L_{2}-\alpha_{2})^{2}}\right)^{\frac{1}{3}}-\left(\frac{\alpha_{1}}{(L_{1}-\alpha_{1})^{2}}\right)^{\frac{1}{3}}}{\left(\frac{\alpha_{1}}{(L_{1}-\alpha_{1})^{2}}\right)^{\frac{1}{3}}+\left(\frac{\alpha_{2}}{(L_{2}-\alpha_{2})^{2}}\right)^{\frac{1}{3}}}\right)^{\frac{1}{4}}\\
 & =\frac{\alpha_{1}}{(L_{1}-\alpha_{1})^{2}}\left(\frac{2\left(\frac{\alpha_{2}}{(L_{2}-\alpha_{2})^{2}}\right)^{\frac{1}{3}}}{\left(\frac{\alpha_{1}}{(L_{1}-\alpha_{1})^{2}}\right)^{\frac{1}{3}}+\left(\frac{\alpha_{2}}{(L_{2}-\alpha_{2})^{2}}\right)^{\frac{1}{3}}}\right)^{\frac{1}{4}}\\
 & +\frac{\alpha_{2}}{(L_{2}-\alpha_{2})^{2}}\left(\frac{2\left(\frac{\alpha_{1}}{(L_{1}-\alpha_{1})^{2}}\right)^{\frac{1}{3}}}{\left(\frac{\alpha_{1}}{(L_{1}-\alpha_{1})^{2}}\right)^{\frac{1}{3}}+\left(\frac{\alpha_{2}}{(L_{2}-\alpha_{2})^{2}}\right)^{\frac{1}{3}}}\right)^{\frac{1}{4}}\\
 & >0.
\end{align*}
\end{proof}
\begin{thm}
Let $L_{\nu}-\alpha_{\nu}>0$ for both groups and $(L_{1}-\alpha_{1})(L_{2}-\alpha_{2})=\bar{L}^{2}$.
Then for all $K,L=\mathbb{N}_{0}$, the moments $\mathbb{E}(\frac{1}{N_{1}^{3K/4}}\sum_{i=1}^{N_{1}}X_{i}\frac{1}{N_{2}^{3K/4}}\sum_{j=1}^{N_{2}}Y_{j})$
are asymptotically given by the expression

\[
\left[\frac{12}{\alpha_{1}(L_{2}-\alpha_{2})^{2}+\alpha_{2}(L_{1}-\alpha_{1})^{2}}\right]^{\frac{K+L}{4}}(L_{1}-\alpha_{1})^{\frac{L}{2}}(L_{2}-\alpha_{2})^{\frac{K}{2}}\frac{\Gamma(\frac{K+L+1}{4})}{\Gamma(\frac{1}{4})}\alpha_{1}^{\frac{K}{4}}\alpha_{2}^{\frac{L}{4}}
\]
if $K+L$ is even and zero otherwise. 
\end{thm}

\begin{proof}
We calculate for any $\underline{i}\in\prod_{k}^{(K)},\underline{j}\in\prod_{l}^{(L)}$:

\begin{align*}
\left|\mathbb{E}(X(\underline{i},\underline{j}))\right| & \leq c\frac{1}{N^{\frac{k+l}{4}}},\\
w_{K}(\underline{i})w_{L}(\underline{j}) & \leq cN_{1}^{\frac{k+K}{2}}N_{2}^{\frac{l+L}{2}}.
\end{align*}

The symbol $c$ in the above inequalities stands for some constant
(not necessarily the same in both lines).

Therefore, each summand is bounded above by

\begin{align*}
\frac{1}{N_{1}^{3K/4}N_{2}^{3L/4}}w_{K}(\underline{i})w_{L}(\underline{j})\left|\mathbb{E}(X(\underline{i},\underline{j}))\right| & \leq c\frac{1}{N_{1}^{3K/4}N_{2}^{3L/4}}N_{1}^{\frac{k+K}{2}}N_{2}^{\frac{l+L}{2}}\frac{1}{N^{\frac{k+l}{4}}}\\
 & =c\alpha_{1}^{1/4}\alpha_{2}^{1/4}N_{1}^{-\frac{1}{4}(K-k)}N_{2}^{-\frac{1}{4}(L-l)},
\end{align*}

which goes to 0 as $N\rightarrow\infty$ if $K>k$ or $L>l$. The
only summand that matters asymptotically is the one where both $k=K,l=L$
hold. In this case, we have

\begin{align*}
w_{K}(\underline{i}) & \approx N_{1}^{K},\\
w_{L}(\underline{j}) & \approx N_{2}^{L}
\end{align*}

and

\[
\mathbb{E}(X(\underline{i},\underline{j}))\approx\left[\frac{12}{\alpha_{1}(L_{2}-\alpha_{2})^{2}+\alpha_{2}(L_{1}-\alpha_{1})^{2}}\right]^{\frac{K+L}{4}}(L_{1}-\alpha_{1})^{\frac{L}{2}}(L_{2}-\alpha_{2})^{\frac{K}{2}}\frac{\Gamma(\frac{K+L+1}{4})}{\Gamma(\frac{1}{4})}\frac{1}{N^{\frac{K+L}{4}}}
\]

provided that $K+L$ is even and 0 otherwise. Multiplying these and
the normalising factors, we obtain the result. 
\end{proof}

email: werner.kirsch@fernuni-hagen.de\qquad{}gabor.toth@fernuni-hagen.de 
\end{document}